\documentclass[12pt]{amsart}
\usepackage{graphicx}
\usepackage{amssymb}
\usepackage{amsmath}
\usepackage{amsthm,amsfonts,bbm}
\usepackage{amscd}
\usepackage{geometry}
\usepackage[all,2cell]{xy}
\usepackage{epsfig,epstopdf}
\usepackage{mathrsfs}

\UseAllTwocells \SilentMatrices

\newtheorem{thm}{Theorem}[section]
\newtheorem{cor}[thm]{Corollary}
\newtheorem{lem}[thm]{Lemma}
\theoremstyle{definition}
\newtheorem{defi}[thm]{Definition}
\theoremstyle{remark}

\numberwithin{equation}{section}
\numberwithin{figure}{section}
\def\A{\mathcal{A}}

\def \ex{\mbox{\rm ex}}
\def \E{\tilde{E}}

\def\lin{\text{lin}}

\def \spex{\mbox{\rm spex}}
\def \T{\text{T}}

\begin{document}
\title[Spectral Tur\'an Problems]{Linear spectral Tur\'an problems for expansions of graphs with given chromatic number}

\author[C.-M. She]{Chuan-Ming She}
\address{School of Mathematical Sciences, Anhui University, Hefei 230601, P. R. China}
\email{cm-she@stu.ahu.edu.cn}

\author[Y.-Z. Fan]{Yi-Zheng Fan*}
\address{Center for Pure Mathematics, School of Mathematical Sciences, Anhui University, Hefei 230601, P. R. China}
\email{fanyz@ahu.edu.cn}
\thanks{*The corresponding author.
Supported by National Natural Science Foundation of China (Grant No. 11871073).
}

\author[L. Kang]{Liying Kang$^\dag$}
\address{Department of Mathematics, Shanghai University, Shanghai 200444, P. R.  China}
\email{lykang@shu.edu.cn}
\thanks{$^\dag$Supported by National Natural Science Foundation of China (Grant No. 11871329).}

\author[Y. Hou]{Yaoping Hou$^\ddag$}
\address{College of Mathematics and Statistics, Hunan Normal University,
Changsha 410081, P. R. China}
\email{yphou@hunnu.edu.cn}
\thanks{$^\ddag$Supported by National Natural Science Foundation of China (Grant No. 11971164).}



\begin{abstract}
An $r$-uniform hypergraph is linear if every two edges intersect in at most one vertex.
The $r$-expansion $F^{r}$ of a graph $F$ is the $r$-uniform hypergraph obtained from $F$ by
enlarging each edge of $F$ with a vertex subset of size $r-2$ disjoint from the vertex set of $F$ such that
distinct edges are enlarged by disjoint subsets.
Let $\ex_{r}^{\lin}(n,F^{r})$ and $\spex_{r}^{\lin}(n,F^{r})$ be the maximum number of edges and the maximum spectral radius of all $F^{r}$-free linear $r $-uniform hypergraphs with $n$ vertices, respectively.
In this paper, we present the sharp (or asymptotic) bounds of $\ex_{r}^{\lin}( n,F^{r})$ and $\spex_{r}^{\lin}(n,F^{r})$ by establishing the connection between the spectral radii of linear hypergraphs and those of their shadow graphs, where $F$ is a $(k+1)$-color critical graph or a graph with chromatic number $k$.
\end{abstract}

\maketitle

\section{Introduction}
A \emph{hypergraph} $H=(V(H),E(H))$ consists of a vertex set $V(H)$ and an edge set $E(H)$, where each edge $e \in E(H)$ is a subset of $V(H)$.
If each edge $e$ of $H$ is an $r$-element subset  of $V(H)$, then
$H$ is called \emph{$r$-uniform}.
A hypergraph is \emph{linear} if any two edges intersect at most one vertex.
So, a simple graph is a linear $2$-uniform hypergraph.

Given family $\mathcal{F}$ of hypergraphs, we say a hypergraph $G$ is \emph{$\mathcal{F}$-free} if it does not have a (not necessarily induced) subgraph isomorphic to any graph $F\in \mathcal{F}$.
Let $\ex_{r}(n,\mathcal{F})$ and $\spex_{r}(n,\mathcal{F})$ denote the maximum number of edges and the maximum spectral radius of all $\mathcal{F}$-free $r$-uniform hypergraphs on $n$ vertices, respectively.
For brevity, we write $\ex_{r}(n,F)$ and $\spex_{r}(n,F)$ instead of $\ex_{r}(n,\{F\})$ and $\spex_{r}(n,\{F\})$ when $\mathcal{F}=\{F\}$.
When considering only simple graphs, we abbreviate them as $\ex(n,F)$ and $\spex(n,F)$ respectively.

The number 	$\ex_{r}(n,\mathcal{F})$ is called \emph{Tur\'an number}, and the corresponding Tur\'an problem of determining the Tur\'an number of graphs or hypergraphs is one of the most fundamental problems in extremal combinatorics \cite{furedi2013history}.
The study of Tur\'an Problem can be dated back at
	least to 1907~\cite{mantel1907problem}, when Mantel proposed the Mantel Theorem that $e(G)\le \left \lfloor n^{2}/4  \right \rfloor $ for every triangle-free $G$ on $n$ vertices, where $e(G)$ denotes the number of edges of $G$.
Tur\'an~\cite{turaan1941extremal} determined $\ex(n,K_{t+1})$, and proved the well-known Tur\'an theorem: the Turán graph, namely the complete $t$-partite graph on $n$ vertices with the part sizes as equal as possible, is the unique $K_{t+1}$-free graph with the maximum number of edges, where $K_{t+1}$ is a complete graph on $t+1$ vertices.

A graph $G$ is \emph{$k$-colorable} if there exists a mapping $\phi: V(G) \to \{ 1,2,\cdots,k\}$ satisfying $\phi (u)\ne \phi(v)$ for any two adjacent vertices $u$ and $v$.
The \emph{chromatic number} of $G$, denoted by $\chi(G)$, is defined as the minimum number $k$ such that $G$ is $k$-colorable.
A graph $G$ is \emph{$(k+1$)-color critical} if $\chi(G)=k+1$, and there exists an edge $e$ of $G$ such that $\chi(G-e)=k$, where $G-e$ denotes the graph obtained from $G$ by deleting the edge $e$.
By the stability theorem, Simonovits \cite{simonovits1968method} determined $\ex(n,F)$ when $F $ is a color critical graph.

 The spectral Tur\'an problem is a spectral version of Tur\'an problem, namely, determining $\spex_r(n,\mathcal{F})$ or the maximum spectral radius of an $\mathcal{F}$-free graph or hypergraph on $n$ vertices.
  Brualdi and Solheid~\cite{brualdi1986spectral} proposed the problem of maximizing the spectral radius over a class of graphs with prescribed structural property in 1986.
 Nikiforov ~\cite{nikiforov2007bounds} presented a spectral version of Turán theorem in 2007, namely, the Tur\'an graph is also the unique $K_{t+1}$-free graph with the maximum spectral radius.
    Nikiforov \cite{nikiforov2009spectral} also obtained a spectral Erd\"os-Stone-Bollob\'as theorem, which implies the Erd\"os-Stone-Bollob\'as and Erd\"os-Stone-Simonovits theorem.

The propose of this paper is to consider the spectral Turán problems in hypergraphs.
Let $F$ be a graph.
The \emph{$r$-expansion} of a graph $F$ is the $r$-uniform hypergraph $F^{r}$ which is obtained from $F$ by inserting $r-2$ additional vertices to each edge of $F$.
Formally, for each edge $e \in E(F)$, we associate it with an $(r-2)$-set $S_e$ with the property:
$S_e \cap V(F)=\emptyset$ for any $e \in E(F)$, and $S_e \cap S_f=\emptyset$ for any two edges $e,f \in E(F)$.
The vertex set and edge set of $F^r$ is defined as follows:
$$ V(F^r)=V(F) \cup (\cup_{e \in E(F)} S_e),~ E(F^r)=\{ e \cup S_e: e \in E(F)\}.$$
 Obviously, $F^{r}$ is a linear $r$-uniform hypergraph which has exactly $|E(F)|$ edges and $|V(F)|+(r-2)|E(F)|$ vertices.

 In 2006, Mubayi~\cite{Mubayi2006A} determined the asymptotic value of $ \ex_{r}(n,K_{t+1}^{r})/\binom{n}{r}$.
 Combining the above result and the supersaturation technique of Erd\'os~\cite{erdos1964extremal}, we will give the asymptotic value of $\ex_{r}(n,F^{r})$ for all $(t+1)$-color critical graph $F$ with $t\ge r$.
 In 2007, Mubayi and Pikhurko~\cite{mubayi2007new} determined
 $\ex_{r}(n,\text{Fan}^{r})$ for an $r$-uniform hypergraph $\text{Fan}^{r}$ which is a generalization of the triangle (as a simple graph).

 When focused on linear hypergraphs, we denote $\ex^{\lin}_{r}(n,\mathcal{F})$ and $\spex^{\lin}_{r}(n,\mathcal{F})$ the maximum number of edges and the maximum spectral radius of all linear $\mathcal{F}$-free $r$-uniform hypergraphs on $n$ vertices, respectively.
 In 2020, F\"uredi and Gy\'arf\'as~\cite{furedi2020extension} considered the problem of determing $\ex_{r}^{\lin}(n,\text{Fan}^{r})$.
  In 2021, Gao and Chang~\cite{gao2021linear} gave an upper bound for $\ex_{r}^{\lin}(n,K_{s,t}^{r})$, where $K_{s,t}$ denotes a complete bipartite graph with two parts of size $s,t$ respectively.
  For more investigations of the hypergraph Tur\'an problem, we recommend readers the survey of Mubayi~\cite{mubayi2016survey}.
	
	The spectral Tur\'an problem of hypergraphs has attracted a great deal of attention recently.
In 2014, Keevash, Lenz and Mubayi~\cite{keevash2014spectral} gave two
	general criteria under which spectral extremal results may be deduced from ‘strong stability’ forms
	of the corresponding extremal results.
In \cite{hou2021spectral}, Hou, Chang and Cooper gave the maximum spectral radius of linear $r$-uniform hypergraphs with forbidden $\text{Fan}^{r}$ under some conditions, and an upper bound for the spectral radius of linear $r $-uniform hypergraphs without Berge $C_4$.
Gao, Chang and Hou~\cite{gao2022spectral} determined the value $\spex_{r}^{\lin}(n,K_{r+1}^{r})$ and gave the extremal graph attained at the maximum spectral radius, which is a transversal design $T(n,r)$.
	
In this paper, we mainly consider the linear spectral Tur\'an problems for the expansion of $(k+1)$-color critical graph $F$.
As $K_{r+1}$ is a $(r+1)$-color critical graph, our results will generalize the related results in \cite{gao2022spectral}.
The main results in this paper are as follows.
	
\begin{thm}\label{sp-crit}
	Let $F$ be a $(k+1) $-color critical graph and $ F^{r}$ be the $r $-expansion of $F$, where $k \ge r $.
Then, for sufficiently large $n$,
\begin{equation} \label{spcri}\spex_{r}^{\lin}(n,F^{r} ) \le\frac{n(k-1)}{k(r-1)}. \end{equation}
If $n=km$ and $m,k,r$ satisfy (\ref{Cond2}) in Theorem \ref{T-Cond2}, the inequality in (\ref{spcri}) can hold as equality for sufficiently large $m$.
\end{thm}

\begin{thm}\label{sp-gen}
	 Let $H$ be a connected linear $r$-uniform hypergraph.
Then
\begin{equation} \label{spec}
\rho (H)\le \frac{n-1}{r-1},
\end{equation}
with equality if and only if $H$ is a $2$-$(n,r,1)$ design.
If $n,r$ satisfy (\ref{Cond1}) in Theorem \ref{T-Cond1}, a $2$-$(n,r,1)$ design exists for sufficiently large $n$.
\end{thm}

It is known that for an $r$-uniform hypergraph $H$ with $n$ vertices and $e(H)$ edges,
\begin{equation}\label{sp-ed} \rho(H) \ge \frac{r e(H)}{n}.\end{equation}
So we can have the following corollary immediately.
	
\begin{cor}	\label{C-cri}
Let $F$ be a $(k+1)$-color critical graph with $k \ge r$.
Then, for sufficiently large $n$,
\begin{equation}\label{cri}
\ex_{r}^{\lin}(n,F^{r} )\le\frac{n^{2} (k-1)}{kr(r-1)}.
 \end{equation}
If $n=km$ and $m,k,r$ satisfy (\ref{Cond2}) in Theorem \ref{T-Cond2}, the inequality in (\ref{cri}) can hold as equality for sufficiently large $m$
\end{cor}	
	
\begin{cor}\label{C-edge}	Let $H$ be a linear $r $-uniform hypergraph. Then
	\begin{equation}\label{edge}
  e(H)\le \frac{n(n-1)}{r(r-1)},
   \end{equation}
with equality if and only if $H$ is a $2$-$(n,r,1)$ design.
If $n,r$ satisfy (\ref{Cond1}) in Theorem \ref{T-Cond1}, a $2$-$(n,r,1)$ design exists for sufficiently large $n$.
\end{cor}

As an application of the above results, we have the following results.	

\begin{lem}\label{N-1}	 Let $k,l,r$ be positive integers such that $k \ge 3$, $l\ge 2$ and $r \ge 2$, and let $c$ be a positive real number.
Then there exists an $n_0=n_{0}(c,k,l,r)$ such that if $H$ is
a linear $r$-uniform hypergraph on $n \ge n_0$ vertices satisfying
\begin{equation} \rho(H)\ge \frac{n}{r-1}\left(1-\frac{1}{k-1}+c\right), \end{equation}
then $H$ contains a $K_{k}(l,\ldots,l)^{r}$.
\end{lem}
	
\begin{thm}\label{N-2}
 Let $F$ be a graph with chromatic number $k$ with $k \ge r+1 \ge 3$.
Then
 \begin{equation} \spex_{r}^{\lin}(n,F^{r})=n\left(\frac{1}{r-1}\left(1-\frac{1}{k-1}\right)+o(1)\right).
 \end{equation}
\end{thm}
	
	The rest of this paper is organized as follows.
In Section 2, we introduce the eigenvalues and eigenvectors of a tensor, the hypergraphs constructed by designs, and some necessary lemmas for the latter discussion. We prove the main results in Section 3.

\section{Preliminaries}
\subsection{Tensors and hypergraphs}
For positive integers $n,r$, a complex \emph{tensor} (also called hypermatrix) $\mathcal{A} =( a_{i_{1}i_{2}\dots i_{r}})$ of order $r$ and dimension $n$ refers to a multidimensional array with entries $a_{i_{1}i_{2}\ldots i_{r}   }\in \mathbb{C}$ for all $i_{1},i_{2},\dots,i_{r}\in [n]:= \{ 1,2,\ldots,n\}$. A tensor $\mathcal{A}$ is called \emph{symmetric}
	if its entries are invariant under any permutation of their indices.
	In 2005, Qi~\cite{qi2005eigenvalues} and Lim~\cite{lim2005singular} independently introduced the eigenvalues of tensors.
If there exists a number $\lambda \in \mathbb{C}$ and a nonzero vector $x \in \mathbb{C}^{n}$	such that
	\begin{equation}\label{ev}\mathcal{A} x^{r-1}=\lambda x^{[r-1] },
\end{equation}
then $\lambda$ is called an \emph{eigenvalue} of $\mathcal{A}$, and ${x}$ is called an
    \emph{eigenvector} of $\mathcal{A}$ corresponding to the eigenvalue $\lambda$,
    where  $\mathcal{A} {x}^{r-1}$ is a vector in $\mathbb{C}^{n}$ whose $i$-th component is given by
    \[  (\mathcal{A} {x}^{r-1} )_{i} =\sum_{i_{2},\ldots,i_{r}=1  }^{n} a_{ii_{2}\dots i_{r}} x_{i_{2} }\cdots x_{i_{r}}, i\in [n],\]
and $x^{[r-1]}:=(x_1^{r-1},\cdots,x_n^{r-1})$.
The \emph{spectral radius} of $\A$, denoted by $\rho(\mathcal{A})$, is the maximum modulus of the eigenvalues of $\mathcal{A}$.

The Perron-Frobenius theorem for nonnegative matrices was generalized to nonnegative tensors by Chang et al. \cite{chang2008perron}, Friedland et al.
\cite{friedland2013perron} and Yang et al. \cite{YY2011-2}.
Here we list part of the theorem, where the weak irreducibility of nonnegative tensors can be referred to \cite{friedland2013perron}.

 \begin{thm}[\cite{chang2008perron,friedland2013perron,YY2011-2}]\label{PF}
 Let $\mathcal{A}$ be a nonnegative tensor of order $r$ and dimension $n$.
 Then the following statements hold.

   \begin{itemize}
    \item[(1)] $\rho(\mathcal{A})$ is an eigenvalue of $\mathcal{A}$ corresponding to a nonnegative eigenvector.

    \item[(1)] If furthermore $\mathcal{A}$ is weakly irreducible, then
    $\rho(\mathcal{A})$ is the unique eigenvalue of $\mathcal{A}$ corresponding to a unique positive eigenvector up to a scalar.

    \end{itemize}
 \end{thm}

Let $H$ be an $r$-uniform hypergraph on
$n$ vertices $v_{1},v_{2},\ldots,v_{n}$.
    In 2012, Cooper and Dutle~\cite{cooper2012spectra} introduced the \emph{adjacency tensor} $\mathcal{A}(H)$ of $H$, which is an order $r$ dimension $n$ tensor
    whose $(i_{1},i_{2},\ldots,i_{r})$-entry is given by
    \[  a_{i_{1}i_{2}\dots i_{r}}=\begin{cases}
    	\frac{1}{(r-1 )!},  & \text{ if } \{ v_{i_{1} },v_{i_{2} },\dots v_{i_{r} }    \}\in E(H ),   \\
    	~~0,  & \text{ otherwise }.
    \end{cases}      \]

    Clearly, $\A(H)$ is nonnegative and symmetric, and it is weakly irreducible if and only if $H$ is connected (\cite{friedland2013perron, YY2011-2}).
The eigenvalues, eigenvectors of $H$ are referring to those of $\A(H)$.
In particular, denote by $\rho(H)$ the spectral radius of $H$ (or $\A(H)$).

   By Theorem \ref{PF}, if
   	$H$ is connected, then there exists a unique positive eigenvector $x$ up to a scalar corresponding to $\rho(H)$, called the \emph{Perron vector} of $H $.
    In addition, $\rho (H)$ is the optimal value of the following
   	maximization (\cite{qi2005eigenvalues})
   	\begin{equation}\label{max} \rho (H )=\max_{\|x\|_{r}=1}  x^\top \mathcal{A}(H)x^{r-1},
   \end{equation}
   and the optimal vector $x$ such that $\rho(H)=\A x^{r}$ is exactly the Perron vector of $H$,
   	where
   	\begin{equation}\label{r-form}  x^\top \mathcal{A}(H)x^{r-1}=\sum_{i_1,\ldots, i_r} a_{i_1,\ldots,i_r} x_{i_1}\cdots x_{i_r}=r \sum_{\{v_{i_1},\ldots,v_{i_r}\} \in E(H)}x_{v_{i_1}}\cdots x_{v_{i_r}}.
   \end{equation}
When $r=2$, the Eqs. (\ref{max}) and (\ref{r-form}) are exactly the expressions of the spectral radius and the quadratic form of a graph respectively.

   	\begin{lem}[\cite{cooper2012spectra}]\label{deg}
    Let $H$ be an $r$-uniform hypergraph. Let $\bar{d}$ be the
   	average degree of $H$ and $\Delta$ be the maximum degree of $H$.
   Then
   	\[ \bar{d} \leq \rho (H) \leq \Delta. \]
   	In particular, if $H$ is a $d$-regular hypergraph, then $\rho (H)=d$.
\end{lem}
   
One can refer \cite{FHB,FHBproc} for more about the spectral radius and its associated eigenvectors of nonnegative tensors and hypergraphs.   	

\subsection{Designs}
We will introduce two kinds of designs in the following.

  \begin{defi}[\cite{wilson1975existence}]
  A \emph{pairwise balanced design} of index $\mu$ (in brief, a $\lambda$-PBD) is a pair $(X,\mathcal{B})$ where $X$ is a set of points, $\mathcal{B}$ is a family of subsets of $X$ (called blocks), such that each $B\in \mathcal{B}$ contains at least two points, any
   	$2$-subset $\{x,y\}$ of $X$ are contained in exactly $\mu$ blocks $B$ of $\mathcal{B}$.
   \end{defi}

  A $\lambda$-PBD on $n$ points in which all blocks have the same cardinality $r$ is traditionally called a \emph{$(n,r,\mu)$-BIBD} (balanced incomplete block design), or a $2$-$(n,r,\mu)$ design.
   	
   \begin{defi}[\cite{mohacsy2011asymptotic}]
   Let $n$, $\mu$, $t$ be positive integers, and $K$ be a set of positive integers.
    A \emph{group divisible $t$-design} (or  $t$-GDD) of order $n$, index $\mu$ and block sizes from  $K$ is a triple $(X,\Gamma,\mathcal{B})$ such that
   	
   	\begin{itemize}

   \item[(1)] $X$ is a set of $n$ elements (called points),
   	
   	\item[(2)] $\Gamma =\{G_{1},G_{2},\cdots, \}$, a set of non-empty subsets of $X$ which partition $X$ (called groups),
   	
   	\item[(3)] $\mathcal{B}$ is a family of subsets of $X$ each of size form $K$ (called
   	blocks) such that each block $B\in \mathcal{B}$ intersects any given group in at most one point,
   	
   	\item[(4)] each $t$-set of points from $t$ distinct groups is contained in exactly $\mu$ blocks.
   \end{itemize}
   \end{defi}
   	
   If a $t$-GDD has $n_{i}$ groups of size $g_{i}$ for  $i \in [l]$, we denote its group type by $g_{1}^{n_{1}}g_{2}^{n_{2}}\ldots g_{l}^{n_{l}}$.
   In particular, a $t$-GDD with group type $m^r$ and block size $r$ (namely $K=\{r\}$) is called a \emph{transversal design}.

   	For a $2$-$(n,r,\mu)$ design, if we treat its points as vertices and  blocks as edges, we can get a hypergraph associated with the design.
   We will not distinguish the a design and a hypergraph associated with the design.
   So, a $2$-$(n,r,1)$ design is a $\frac{n-1 }{r-1 } $-regular linear hypergraph with $\frac{n(n-1)}{r(r-1)}$ edges.
     Similarly, a $2$-GDD of group type $m^{k}$ with block size $r$ and index $1$ is
  a $\frac{m(k-1)}{r-1}$-regular hypergraph with $\frac{m^{2}k (k-1) }{r(r-1)}$ edges.

   An important issue is the question of the existence of the above two designs or the associated hypergraphs.
   The following theorems characterized the existence of the above two designs.

   \begin{thm}[\cite{wilson1975existence}]\label{T-Cond1}
    Given positive integers $n,r$, a $2$-$(n,r,1)$ design exists  for all sufficiently large integers $n$ under the following conditions:
    \begin{equation}\label{Cond1}
     n-1 \equiv 0 \mod{(r-1)}, ~~~n(n-1) \equiv 0 \mod{r(r-1)}.
     \end{equation}
    \end{thm}
   	
   \begin{thm}[\cite{mohacsy2011asymptotic}]\label{T-Cond2}
    Let $r$ and $k$ be positive integers with $2 \le r \le k$.
    Then there exists an integer $m_{0}=m_{0}(r,k)$ such that
   	for any integer $m \ge m_{0}$ there exists a group divisible design of group type $m^{k}$ with block size $r$ and index $1$ satisfying the condition:
    \begin{equation}\label{Cond2}	
     m(k-1) \equiv 0 \mod{r-1},~~~ m^{2}k(k-1)  \equiv 0 \mod{r(r-1)}.
      \end{equation}
    \end{thm}
   	
   \subsection{Spectral radii of graphs and hypergraphs}
   Some of the necessary lemmas on the spectral radii of graphs and hypergraphs
   are presented below for the discussion in Section 3,
   with most of the results coming from Nikiforov.
   We write $\T_{k}(n)$ for the $k$-partite Tur\'an graph of order $n$,
   $ K_{k}(s_{1},\ldots ,s_{k})$ for the complete $k$-partite graph with parts of sizes $s_{1}\ge 2$, $s_{2},\ldots ,s_{k}$, respectively, and $K_{k}^{+}(s_{1},\ldots,s_{k})$ obtained from $K_{k}(s_{1},\ldots,s_{k})$ by adding an edge within the first part.
   By the result in \cite{FLZ},
   \begin{equation}\label{sp-Tur}\rho(\T_{k}(n))\le n\left(1-\frac{1}{k}\right).\end{equation}

   	\begin{lem}[\cite{nikiforov2009spectral}]\label{N-sp}
    Let $n,k$ be positive integers and $c$ be a positive real number such that $k\ge 3 $ and $(c/k^{k} )^{k}\ln{n}\ge 1$.
     If $G $ is a graph with $n$ vertices satisfying
     \[ \rho (G )\ge n\left(1-\frac{1}{k-1}+c\right), \]
   	then $G$ contains a $K_{k}(s,\dots s,t)$ with $s \ge \left \lfloor (c/k^{k}  )^{k}\ln{n} \right \rfloor$ and $t> n^{1-c^{k-1}}$.
   \end{lem}
   	
   \begin{lem}[\cite{nikiforov2007spectral}]\label{N-Tur}
    Let $n,k$ be positive integers and $c$ be a positive real number such that $k\ge 2$, $c=k^{-(2k+9)(k+1)}$ and $n\ge e^{2/c}$.
    If $G$ is a graph on $n$ vertices satisfying
   	$\rho(G)> \rho(\T_{k}(n))$,
   then $G$ contains a
   $K_{k}^{+}(\left \lfloor c\ln{n}\right\rfloor,\ldots,\left \lfloor c\ln{n} \right \rfloor)$.
   \end{lem}
   	
  The key method of this paper is to build a connection between the spectral radius of a hypergraph and that of its shadow graph.
  Let $H$ be a linear $r$-uniform hypergraph.
  The \emph{shadow graph} of $H$, denoted by $\partial H$, is
   	the graph with vertex set $V(\partial H)=V(H)$ and edge set $E(\partial (H)  )=\{\{x,y\}: \{x,y\}\subseteq e \in E(H)\}$.
   Alternatively, $\partial H$ is obtained from $H$ by replacing each edge $e$ with a clique on the vertices of $e$.

   	\begin{lem}\label{Lconn} Let $H$ be a connected linear $r$-uniform graph.
    Then
    \begin{equation}\label{conn} \rho (H )\le \frac{1}{r-1}\rho (\partial H), \end{equation}
     with equality if and only if $H$ is regular.
   \end{lem}

 \begin{proof}
Let ${x}$ be a Perron vector of $H$ with $\| x \|_{r}=1$.
For an edge $e \in E(H)$, denote $x^e=\prod_{v \in e}x_v$.
We have
   	\begin{equation}\label{key}
   \begin{split}
   		\rho (H )&=x^\top \A(H)x^{r-1}=r\sum_{e\in E(H)}x^{e}\\
   &=r\sum_{e\in E(H)} \prod_{\{ i,j \}\in \binom{e}{2}}  (x_{i}x_{j})^{\frac{1}{r-1}}\\
   &\le r\sum_{e\in E(H)} \frac{1}{\binom{r}{2} }\sum_{\{ i,j \}\in \binom{e}{2}  }  (x_{i}x_{j})^{\frac{\binom{r}{2}}{r-1}}\\
    &=\frac{2}{r-1} \sum_{e\in E(H)}  \sum_{\{ i,j \}\in \binom{e}{2}}  (x_{i})^{\frac{r}{2} } (x_{j} )^{\frac{r}{2}}=\frac{2}{r-1} \sum_{\{ i,j\}\in E(\partial H)}  (x_{i})^{\frac{r}{2}}
   (x_{j})^{\frac{r}{2}}\\
   &= \frac{1}{r-1} (x^{[r/2]})^\top \A(\partial H) x^{[r/2]} \\
   &\le \frac{1}{r-1} \rho(\partial H),
   \end{split}
   	\end{equation}
      where the first inequality follows from the inequality for arithmetic means and geometric means, and the second inequality follows from Eq. (\ref{max}) as $\|x^{[r/2]}\|_2=1$.

      If $H$ is regular, then $\partial H$ is also regular as $H$ is linear, and the Perron vector $x=|V(H)|^{-1/r}\mathbf{1}$, where $\mathbf{1}$ is an all one vector.
      It is easy to see all inequalities in (\ref{key}) become equalities in this case, and hence we arrive the equality in Eq. (\ref{conn}).
      On the other hand, if $\rho (H )= \frac{1}{r-1}\rho (\partial H)$,
      then all entries of $x$ are equal from the first inequality in (\ref{key}), implying that $x=|V(H)|^{-1/r}\mathbf{1}$.
      By the eigenvector equation (\ref{ev}), we get that $H$ is regular.
     \end{proof}

   	\section{Proof}
   	
 	\begin{proof}[Proof of Theorem \ref{sp-crit}]
  It may be assumed that $F-e$ is a $k$-partite graph with all part sizes less than $l$, which implies $K_{k}^{+}(l,\ldots,l)$ contains a copy of $F$.
   Let $H$ be an $F^{r}$-free linear hypergraph.
   Suppose that $\rho(H)> \frac{n(k-1) }{k(r-1)}$.
   By Lemma \ref{Lconn} and Eq. (\ref{sp-Tur}), we have
    \[ \rho(\partial H)\ge (r-1 )\rho(H)> \frac{n(k-1)}{k}\ge \rho(\T_{k}(n)  ). \]
    So, by Lemma \ref{N-Tur}, for sufficiently large $n$, $\partial H$ contains a $K_{k}^{+}(\left \lfloor c\ln{n}\right\rfloor,\ldots,\left \lfloor c\ln{n} \right \rfloor)$, where $c$ satisfies the condition in Lemma \ref{N-Tur}.
    Below we will construct a $(K_{k}^{+}(l,\ldots,l))^{r} $, the $r$-expansion of $K_{k}^{+}(l,\dots ,l)$, by a copy $K_{k}^{+}(\left \lfloor c\ln{n}  \right  \rfloor,\ldots ,\left \lfloor c\ln{n} \right \rfloor)$ which is contained in $\partial H$.
    Thus $H$ contains a $F^{r}$; a contradiction.
   	
   	Construction method. Take a copy of $K_{k}^{+}(\lfloor c\ln{n}\rfloor,\ldots, \lfloor c\ln{n} \rfloor)$ in $\partial H$.
   Let $V_{i}=\{v_{i1},v_{i2},\ldots,v_{i,\lfloor c\ln{n} \rfloor}\}$, $i\in [k]$, be all $k$ parts of this copy, and the critical edge $e:=\{v_{11},v_{12}\}$ is contained in $V_{1}$.
    Since $H$ is linear, if $\{x,y\}$ is an edge of $\partial H$, then
   	$\{x,y\}$ must be contained in a unique edge of $H$ denoted by $E_{xy}$.
 Denote $\E_{xy}=E_{xy}\backslash \{x,y\}$ if $\{x,y\} \in E(\partial H)$, and
 $\E_{xy}=\emptyset$ otherwise.

   	First let $A_{1}=\{v_{11},v_{12}\}$ and $B_{1}=\E_{v_{11}v_{12}}$. For $j=2,\ldots,l-1$, take a vertex $v \in V_1\setminus (A_{j-1} \cup B_{j-1})$, let $A_j=A_{j-1} \cup \{v\}$, $B_j =B_{j-1} \cup \{\E_{vw} : w \in A_{j-1}\} $.
   So we have taken an $l$-subset $A_{l-1}$ from $V_1$.

   We are going to take an $l$-subset from each $V_i$ for $i=2,\ldots,k$ in the following way.
      For each $i=2,\ldots, k$, and each $j=(i-1)l,\ldots,il-1$,
   let $C_{j-1}=\cup\{\E_{xy}: x \in A_{j-1},y \in B_{j-1}\}$,
   take a vertex $v \in V_i \backslash (A_{j-1} \cup B_{j-1}\cup C_{j-1})$, and
   let $A_j=A_{j-1} \cup \{v\}$ and $B_j=B_{j-1} \cup \{\E_{vw} : w \in A_{j-1}\}$.
  So we get sets $A_i,B_i$ for $i=1,\ldots,kl-1$, where $A_{il-1}\backslash A_{(i-1)l-1}$ is an $l$-subset taken from $V_i$ for  $i=2,\ldots,k$.

   Note that $|A_{j-1}|=j$, $|B_{j-1}|\le (r-2)\binom{j}{2}$, and $|C_{j-1}| \le (r-2)|A_{j-1}||B_{j-1}|=j(r-2)^2\binom{j}{2}$ for $j=l,\ldots,kl-1$.
   So, for $j=2,\ldots,l-1$, if
   $$ |V_1|=\lfloor c \ln n\rfloor > j+(r-2)\binom{j}{2}\ge |A_{j-1}|+|B_{j-1}|,$$
   and for $i=2,\ldots,l$ and $j=(i-1)l,\ldots,il-1$, if
   $$|V_i|=\lfloor c \ln n\rfloor > j+(r-2)\binom{j}{2}+j(r-2)^2\binom{j}{2} \ge |A_{j-1}|+|B_{j-1}|+|C_{j-1}|,$$
   the above procession can be continued.
   This is guaranteed by taking $n$ sufficiently large.

   Let $G$ be a subgraph of $K_{k}^{+}(\lfloor c\ln{n}\rfloor,\ldots, \lfloor c\ln{n} \rfloor)$ induced by the vertices of $A_{kl-1}$ and $G^r$ be the sub-hypergraph of $H$ induced by edges $E_{xy}$ for $\{x,y\} \in E(G)$.
    Then $G=K_{k}^{+}(l,\ldots,l)$ and $G^r=(K_{k}^{+}(l,\ldots,l))^r$.
The construction is completed.

   	 On the other hand, if $n=mk$, and $m,k,r$ satisfies the condition (\ref{Cond2}) in Theorem \ref{T-Cond2}, then a $2$-GDD of group type $m^{k}$ with block size $r$ and index $1$ exists.
   Let $\hat{H}$ be the hypergraph associated with the above $2$-GDD.
   Then $\hat{H}$ is $F^r$-free and $\frac{n(k-1)}{k(r-1)}$-regular.
   So $\rho(\hat{H})=\frac{n(k-1)}{k(r-1)}$ by Lemma \ref{deg}.
      \end{proof}
   	
   	\begin{proof}[Proof of Theorem \ref{sp-gen}]
   Observe that $\partial H$ is a connected graph on $n$ vertices, and
       $\rho(\partial H)\le n-1$ with equality if and only if $\partial H$ is a complete graph.
   By Lemma \ref{Lconn}, we have
   \[\rho(H)\le \frac{1}{r-1}\rho (\partial H)\le \frac{n-1}{r-1}. \]
   	The above equality holds if and only if $H$ is regular and $\partial H$ is a complete graph, which is equivalent to that $H$ is a $2$-$(n,r,1)$ design.
      \end{proof}

   	\begin{proof}[Proof of Corollary \ref{C-cri}]
     By the inequality (\ref{sp-ed}) and Theorem \ref{spcri},
     \begin{equation}\label{Coro1} \ex_r^{\lin}(n,F^r) \le \frac{n}{r}\spex_r^{\lin}(n,F^r) \le \frac{n^2(k-1)}{kr(r-1)}.\end{equation}
     If  $n=km$ and $m,k,r$ satisfy (\ref{Cond2}) in Theorem \ref{T-Cond2},
     then a $2$-GDD of group type $m^{k}$ with block size $r$ and index $1$ exists.
     The associated hypergraph $\hat{H}$ makes the second inequality of (\ref{Coro1}) holds as equality $\frac{n^2(k-1)}{kr(r-1)}$ edges.
     So the inequality in Corollary  \ref{C-cri}  holds as equality in this case.
     \end{proof}

     \begin{proof}[Proof of Corollary \ref{C-edge}]
     By the inequality (\ref{sp-ed}) and Theorem \ref{sp-gen},
     $$ e(H) \le \frac{n}{r}\rho(H) \le \frac{n(n-1)}{r(r-1)}.$$
     The second inequality holds as equality if and only if $H$ is a $2$-$(n,r,1)$ design.
     Observe that in this case $H$ has exactly $\frac{n(n-1)}{r(r(r-1)}$ edges.
     So we prove the Corollary \ref{C-edge}.
     \end{proof}
        	
   	\begin{proof}[Proof of Lemma \ref{N-1}]
   Since $\rho (H)\ge \frac{1}{r-1}(1-\frac{1}{k-1}+c)n $,
     by Lemma \ref{Lconn} we have
     \[ \rho(\partial H) \ge (r-1) \rho(H)\ge n(1-\frac{1}{k-1}+c). \]
     By Lemma \ref{N-sp}, there exists an $n_0=n_0(c,k)$ such that if $n \ge n_0$, then $\partial H$ contains a $K_{k}(s,\ldots s,t)$ with $s \ge  \lfloor (c/k^{k})^{k}\ln n \rfloor$ and $t> n^{1-c^{k-1}}$.
     By a similar discussion as in Theorem \ref{sp-crit}, there exists an $n'_0=n'_0(c,k,l,r)$ such that if $n \ge n'_0$, we can construct a $ K_{k}(l,\dots ,l)^{r}$ in $H$.
       \end{proof}

   	\begin{proof}[Proof of Theorem \ref{N-2}]
   By Lemma \ref{N-1}, we have
   \[ \limsup_{n \to \infty} \frac{\spex_{r}^{\lin}(n,F^{r})}{n} \le \frac{1}{r-1}\left(1-\frac{1}{k-1}\right).\]

   	   	On the other hand, by Theorem \ref{T-Cond2}, there exists a $m_0=m_{0}(r,k)$ such that if $m\ge m_{0}$ there exists a $2$-GDD of group type $m^{k-1}$ with block size $r$ and index $1$ satisfying
      \[ m(k-2) \equiv 0 \mod{(r-1)}, ~~~m^{2}(k-1)(k-2)  \equiv 0 \mod{r(r-1)}. \]
      So, the hypergraph $H$ associated with the above $2$-GDD on $n=m(k-1)$ vertices is $\frac{n}{r-1} (1-\frac{1}{k-1})$-regular, and then
   $\rho(H)= \frac{n}{r-1}(1-\frac{1}{k-1})$.
   Obviously, $H$ is $F^{r}$-free, which implies that
   \[ \liminf_{n \to \infty} \frac{\spex_{r}^{\lin}(n,F^{r})}{n} \ge  \frac{1}{r-1}(1-\frac{1}{k-1}).\]
   The result follows.
   \end{proof}

\end{document}